%% file: bcdual-hwh-30-04-2014.tex
\documentclass[10pt]{amsart}

\numberwithin{equation}{section} 

\usepackage{fancyhdr}
\usepackage[stable]{footmisc}
\usepackage{bm}
\usepackage{amssymb}
\usepackage{xypic}
\usepackage{graphicx}

\input{macrospicard}

\def\WW{{{\mathbb{W}}}}
\def\HV1{{\FF_9[u^{\pm 1}]}}

\date{\today} 
\begin{document}

\title{The Brown-Comenetz dual of the $K(2)$-local sphere 
at the prime $3$} 

\begin{abstract}
We calculate the homotopy type of the Brown-Comenetz dual $I_2$ of the $K(2)$-local sphere at the prime $3$ and show that there is an equivalence in the $K(2)$-local category between $I_2$ and a smash product of the determinant twisted sphere and an exotic element $P$ in the Picard group.  We give a characterization of $P$ as well. A secondary aim of the paper is extend our library of calculations in the $K(2)$-local category.
\end{abstract}
 
\author{Paul G. Goerss}
\address[Paul Goerss]{Department of Mathematics, 
Northwestern University, Evanston, IL 60208, U.S.A. 1-847-491-8544}
\email{poerss@math.northwestern.edu}

\author{Hans-Werner Henn}
\address[Hans-Werner Henn]{Institut de Recherche Math\'ematique Avanc\'ee,
C.N.R.S. - Universit\'e de Strasbourg, F-67084 Strasbourg,
France}
\email{henn@math.unistra.fr}

\thanks{The first author was partially supported by the National Science 
Foundation (USA) and the second author was partially supported by ANR ``HGRT''.}

\maketitle

\section{Introduction}

Because $\QQ/\ZZ$ is an injective abelian group, the functor
$$
X \mapsto \Hom(\pi_\ast X,\QQ/\Z)
$$
defines a cohomology theory on spectra represented by a spectrum $I$;
the Brown-Comenetz dual of $X$ is then the function spectrum 
$IX = F(X,I)$. In particular, $I$ itself is the Brown-Comenetz
dual of the sphere spectrum. This duality on spectra was introduced
in \cite{BC} and some of the basic properties are outlined there. 
Here we are interested in Brown-Comenetz duality
for the $K(n)$-local category. There are a number of reasons to
restrict to this category; for example, the work of Gross and Hopkins
\cite{HopGr} indicates that in the $K(n)$-local category, Brown-Comenetz
duality is a homotopical analog of Grothendieck-Serre duality.

In order
to discuss our results we need a bit of notation. Fix a prime $p$.
Let $K(n)$ be the $n$th Morava $K$-theory at that prime and let
$E_n$ be the associated Lubin-Tate theory. Both theories are
complex orientable; to be concrete, we specify that the formal
group over $K(n)_\ast$ is the Honda formal group $\Gamma_n$ of height $n$,
and the formal group over $(E_n)_\ast$ is a $p$-typical choice of the
universal deformation of the Honda formal group. We write
$$
(E_n)_\ast X \defeq \pi_\ast L_{K(n)}(E_n \wedge X).
$$
For all $X$, $(E_n)_\ast X$ is a twisted $\GG_n$-module,
where
$$
\GG_n = \Aut(\Gamma_n) \rtimes \Gal(\FF_{p^n}/\FF_p) 
$$
is the big Morava stabilizer group. Indeed, by the Hopkins-Miller Theorem,
$\GG_n$ acts on the spectrum $E_n$ itself; this induces the 
action on homology. Some care must be taken to get the topology on
$E_\ast X$ precise; we will take this point up in Section 2.
If $p$ and $n$ are understood we may also write $E$ for $E_n$
and $E_\ast X$ for $(E_n)_\ast X$. We write $L_{K(n)}$ for
localization with respect to $K(n)$.

If $X$ is a $K(n)$-local spectrum, then $IX$ need not be $K(n)$-local. To address
this, let $L_{n}$ be localization at $K(0) \vee \ldots \vee K(n)$
and define the $n$th monochromatic layer $M_nX$ of $X$
to be the fiber of $L_n X \to L_{n-1}X$. Then, for $K(n)$-local $X$,
we define the Brown-Comenetz dual to be
$$
I_n X = F(M_nX,I) = IM_nX
$$
This version of Brown-Comenetz duality was first
laid out by Hopkins (see \cite{PicFirst}). The basic properties are
worked out in \S 10 of \cite{HS}.

Define $I_n = I_nL_{K(n)}S^0$. 
One of the main results of \cite{HopGr} (see also \cite{StricklandGH}) asserts that
\begin{equation}\label{GrHopFor}
(E_n)_\ast I_n \cong \Sigma^{n^2-n} (E_n)_\ast \langle \det \rangle
\end{equation}
where $(E_n)_\ast \langle \det \rangle$ is the 
twisted $\GG_n$-module
obtained by twisting $(E_n)_\ast S^0$ by a determinant action. This
action is explained in detail in Section 5.
 
In particular, $(E_n)_\ast I_n$ is a free module of rank 1 over $(E_n)_\ast$
and, hence, by the results of \cite{PicFirst} and \cite{HovSad}, $I_n$ is an element of the
Picard group $\Pic_n$ of invertible elements in the $K(n)$-local
category. The formula (\ref{GrHopFor})
determines the homotopy type of $I_n$ for large primes or, indeed, for
any prime for which $(E_n)_\ast(-)$ differentiates the elements of the
Picard group. In this case, there is an equivalence
in the $K(n)$-local category
$$
I_n \simeq \Sigma^{n^2-n} \Sdet
$$
where $\Sdet$ is the determinant twisted sphere. The construction of
$\Sdet$ is reviewed in section 6; it has the property that $(E_n)_\ast \Sdet
\cong (E_n)_\ast\langle\det\rangle$. 

More information is needed when the $K(n)$-local Picard group
has exotic elements. Let  $\kappa_n$ be the subgroup of the $K(n)$-local
Picard group with elements the weak equivalence classes of
invertible spectra $X$ so that
$(E_n)_\ast X \cong (E_n)_\ast S^0$ as twisted $\GG_n$-modules.
If $\kappa_n \ne 0$ the computation of $I_n$ is no longer a purely
algebraic problem. For example, if $p=2$ and $n=1$ it is fairly simple to 
show that $I_1 \wedge V(0) \simeq \Sigma^{-2}L_{K(1)}V(0)$, where $V(0)$ is
the mod $2$ Moore spectrum. Since $\Sdet = S^2$ in this case, there must
be some contribution from $\kappa_1$. Using the results of
 \cite{HMS}, one can show $I_1 = \Sigma^2 L_{K(1)}DQ$ where
$L_{K(1)}DQ \in \kappa_1 \cong \ZZ/2$ is the generator. The spectrum
$L_{K(1)}DQ$ 
is the $K(1)$-localization of the ``dual question mark complex"; it
is characterized by the fact that it is in $\kappa_1$ and
$KO\wedge DQ \simeq \Sigma^4 KO$. This example is discussed in
some detail in \cite{HovSad} \S 6.

A similar phenomena occurs at $p=3$ and $n=2$, but it is harder to
characterize the exotic element in $\kappa_2$. Let $V(1)$ be
the four cell complex obtained as the cofiber of the Adams map
$\Sigma^4 V(0) \to V(0)$ of the mod $3$ Moore spectrum. The spectrum
$V(1)$ is a basic example of a type $2$ complex. Our main results
can be encapsulated in the following theorem. Here and below we will write $X \wedge Y$
for $L_{K(n)}(X \wedge Y)$ whenever we work in the $K(n)$-local
category. This is the smash product internal to the $K(n)$-local category.

\begin{thm}\label{BCofS} Let $n=2$ and $p=3$. There is a unique
spectrum $P \in \kappa_2$ so that $P \wedge V(1) \simeq \Sigma^{48}L_{K(2)}V(1)$.
The Brown Comenetz dual $I_2$ 
of $L_{K(2)}S^0$ is given as 
$$
I_2\simeq S^2\wedge \Zdet \wedge P\ . 
$$ 
Furthermore $I_2 \wedge V(1) \simeq \Sigma^{-22}L_{K(2)}V(1)$.
\end{thm} 

The calculation that $\pi_\ast I_2 \wedge V(1) \cong \pi_\ast\Sigma^{-22}V(1)$
was known to Mark Behrens (see \cite{BehTop06}) and is the starting point
of our argument. The spectrum $P$ is discussed in Theorem \ref{P-define} below. 
The other parts of the proof of Theorem \ref{BCofS} are spread over Theorem \ref{BCofSproof}
and Corollary \ref{dual-v1-real}.

Let $G_{24} \subseteq \GG_2$ be a maximal finite subgroup containing an element 
of
order $3$. Then it is true that $E_2^{hG_{24}} \wedge P \simeq \Sigma^{48}
E^{hG_{24}}$.
Since $E_2^{hG_{24}}$ is periodic of period $72$, this is analogous to the 
statement at  $n=1$ and $p=2$ that $KO\wedge DQ \simeq \Sigma^4 KO$.
However, this is not enough to characterize the spectrum
$P$. In \cite{Picat3} we constructed
non-trivial elements $Q \in \kappa_2$ so that $E_2^{hG_{24}} \wedge Q
\simeq  E_2^{hG_{24}}$. If we fix one such $Q$, then it follows from
the main theorem of \cite{Picat3} that every element of $\kappa_2$
may be written $ P^{\wedge a} \wedge Q^{\wedge b}$
with $(a,b) \in \Z/3 \times \Z/3$. It then follows from 
Equation (\ref{GrHopFor}) that we have an equivalence
\begin{equation}\label{GH-pic-for}
I_2 = \Sigma^2 \Sdet \wedge P^{\wedge a} \wedge Q^{\wedge b}
\end{equation}
for some $(a,b) \in \Z/3 \times \Z/3$. In Proposition \ref{v1-fail}
we show that $Q^{\wedge b} \wedge V(1)$ 
has the homotopy groups of a suspension of $L_{K(2)}V(1)$   
only if $b=0$. In Theorem \ref{SdetV(1)}
we show that $\Sdet \wedge V(1)
\simeq \Sigma^{72}L_{K(2)}V(1)$; the calculation of $I_2$ then follows.
As this summary indicates, a secondary aim of this paper is to add to our store
of $K(2)$-local calculations. 

The paper ends with an appendix which displays graphically the homotopy
groups of $E_2^{hG_{24}} \wedge V(1)$ and $E_2^{h\GG_2^1} \wedge V(1)$,
where $\GG_2^1 \subseteq \GG_2$ is the kernel of a reduced determinant map.
(See (\ref{det-defined}).) By \cite{shv1} and \cite{GHM} we know
$$
\pi_\ast L_{K(2)}V(1) \cong 
\Lambda(\zeta)\otimes \pi_\ast E_2^{h\GG_2^1} \wedge V(1)
$$
where $\Lambda(\zeta)$ is an exterior algebra on a class of degree $-1$.
\bigskip

{\bf Acknowledgements:} We have had many helpful conversations about
this and related material with Mark Mahowald and Charles Rezk.

\section{Background}

We are working in the $K(n)$-local category and all our
spectra are implicitly localized. In particular, we emphasize that
$X \wedge Y = L_{K(n)}(X \wedge Y)$, as this is the smash product
internal to the  $K(n)$-local category. In addition, we will write
$V(1)$ for $L_{K(2)}V(1)$; this will greatly economize notation in calculation
and in  the statements of our results. If we work with the unlocalized 
version of $V(1)$, we will explicitly say so. That said, we will write
$L_{K(2)}S^0$ for the localized sphere.

We will write $E$ for
$E_n$, where $E_n$ is the Lubin-Tate spectrum with
$$
E_\ast = (E_n)_\ast \cong W(\FF_{p^n})[[u_1,\cdots,u_{n-1}]][u^{\pm1 }]
$$
with $u_i$ in degree $0$ and $u$ in degree $-2$. Note that
$E_0$ is a complete local ring with residue field $\FF_{p^n}$;
the formal group over $E_\ast$ is a choice of universal deformation
of the Honda formal group $\Ga_n$ over $\FF_{p^n}[u^{\pm 1}]$. 
We make the $p$-typical choice of the coordinate of the deformation
so that $v_i = u_iu^{-p^i+1}$, $1 \leq i \leq n-1$,
and $v_n = u^{-p^n+1}$. The endomorphism ring of $\Ga_n$ is
given by the non-commutative polynomial ring
$$
\End(\Ga_n)\cong W(\FF_{p^n})\langle S \rangle/(S^n-p,Sa=\phi(a)S)
$$
where $S$ is the endomorphism given by $S(x) = x^p$ and
$\phi:W(\FF_{p^n}) \to W(\FF_{p^n})$ is the lift of the Frobenius.
Then the automorphism group $\SS_n = \Aut(\Ga_n)$ is  the group of units
in this ring, and the extended group $\GG_n =    \Aut(\Ga_n) \rtimes \Gal(\FF_{p^n}/\FF_p)$
acts on $E$, by the Hopkins-Miller theorem.  The groups $\SS_n$ and $\GG_n$
are the Morava stabilizer group and the big Morava stabilizer group, respectively.

Since $\GG_n$ acts on $E$, $\GG_n$ acts on $E_\ast X$.
The $E_\ast$-module $E_\ast X$ is equipped with the
${\goth m}$-adic topology 
where ${\goth m}$ is the maximal ideal in $E_0$. This topology 
is always topologically complete, but need not be separated. 
With respect to this topology, the group $\GG_n$ acts through
continuous maps and the action is twisted because it
is compatible with the action of $\GG_n$ on the coefficient ring
$E_\ast$. See \cite{GHMR} \S 2 for
some precise assumptions which guarantee that $E_\ast X$ 
is complete and separated. All modules which will be used in this paper 
will in fact satisfy these assumptions, and we will call these modules
twisted $\GG_n$-modules.

We will work at $p=3$ and $n=2$ exclusively. The right action of $\Aut(\Ga_2)$ on $\End(\Ga_2)$ 
defines a determinant map $\det:\Aut(\Ga_2) \to \ZZ_3^\times$ which extends to a determinant map
\begin{equation}\label{det-defined}
\xymatrix{
\GG_2 \cong  \SS_2 \rtimes \Gal(\FF_9/\FF_3)
\rto^-{\det \times 1} &  \ZZ_3^\times \times \Gal(\FF_9/\FF_3) 
\rto^-{p_2} & \ZZ_3^\times\ .
}
\end{equation}
Let $\GG_2^1 \subseteq \GG_2$ be the kernel of the composition
\begin{equation}\label{G21}
\xymatrix{
\GG_2 \rto & \ZZ_3^\times \rto & \ZZ_3^\times/\{\ \pm 1\}\ .
}
\end{equation}
The center of $\GG_2$ is
isomorphic to $\ZZ_3^\times \subseteq \WW(\FF_9)^\times \subseteq \Aut(\Ga_n)$;
there is a further decomposition $\ZZ_3^\times = U_1 \times \{\pm 1\}$ where
$U_1 = 1 +3\ZZ_3$ and there is a non-canonical isomorphism $U_1 \cong \ZZ_3$.
We fix such an isomorphism (it will not matter in the sequel which one we choose) 
and thus we get a fixed isomorphism 
\begin{equation}\label{GG2-splitting}
\ZZ_3 \times \GG_2^1 \cong U_1 \times \GG_2^1 \cong
\GG_2\ . 
\end{equation}
In addition, $\GG_2$ contains two important finite 
subgroups $G_{24}$ and $SD_{16}$ of orders $24$ and $16$ respectively.
The groups $SD_{16}$ is generated by the Frobenius $\phi$ in
$\Gal(\FF_9/\FF_3)$ and $\omega \in \WW(\FF_9)^\times \subseteq
\Aut(\Ga_2)$, a primitive $8$th root of unity. The group $G_{24}$
is a choice for a maximal finite subgroup containing an element of 
order $3$. Specifically, the element
$$
a= -\frac{1}{2}(1+\omega S)
$$
is our usual 
choice of an element of order $3$. Then $G_{24}$ is generated by
$a$, $\omega^2$, and $\omega\phi$.
All of the these subgroups and the homotopy
fixed points $E^{hF}$ where $F$ is finite
are discussed extensively in \cite{GHM} and \cite{GHMR}.

\begin{rem}\label{id-p-v1} A very basic fact we will need about the Smith-Toda complex
$V(1)$ is that the identity $V(1)\to V(1)$ has order $3$ even before $K(2)$-localization;
indeed $3:V(1) \to V(1)$ factors through $V(1)/V(0) = \Sigma^5V(0)$
and $\pi_5V(1) = \pi_6V(1) = 0$ before localization. This implies
$$
\pi_\ast X \wedge V(1)
$$ 
is an $\FF_3$-vector space for all $X$, localized or not.
\end{rem}

\section{The calculation of $\pi_\ast L_{K(2)}V(1)$}

We recapitulate
and expand the results from \cite{GHM} using a slightly different order of ideas. 
Let  $\GG_2^0 \subseteq \GG_2$ be
the index $2$ subgroup given as the kernel of the composition
$$
\xymatrix{
\GG_2 \rto^-\det & \ZZ_3^\times \rto & C_2 = \{ \pm 1 \}
}
$$
where $\det:\GG_2 \to \ZZ_3^\times$ is the extended  determinant map of (\ref{det-defined})
and the second map is reduction modulo $3$. For 
our calculation of $\pi_\ast \Sdet \wedge V(1)$ in \S 6, it will be helpful to make
calculations with $\GG_2^0$ as well.

\subsection{Cohomological detection by centralizers}
%The cohomology of the Sylow subgroup of %$\SS_2$}
Let %$\SS_2 = \Aut(\Ga_2)$ and  
$S_2 \subseteq \SS_2$ be the $3$-Sylow subgroup;
then 
$$
\GG_2 = S_2 \rtimes SD_{16}
$$
where $SD_{16}$ is generated by $\omega \in \FF_9^\times$ and the Frobenius
$\phi$ in the Galois group. We also have $\GG_2^0 \cong S_2 \rtimes D_8$
where $D_8 \subseteq SD_{16}$ is generated by $\omega^2$ and $\phi$. 

Let $C_{\GG_2}(a)$ be the centralizer in $\SS_2$ of the subgroup $C_3$ generated by our chosen 
element $a\in \GG_2$ of order $3$ and let $N_{\GG_2}(a)$ be its normalizer in $\SS_2$. 
The inclusion $C_{\GG_2}(a)\subseteq N_{\GG_2}(a)$ is of index $2$ and the quotient is 
generated by the image of $\omega^2$. The structure of $C_{\GG_2}(a)$ and of $N_{\GG_2}(a)$ 
has been discussed in Proposition 20 of \cite{HennRes} (see also \cite{HennDuke} and \cite{GHM}). 
If we write $C_n(x)$ for the cyclic group of order $n$ generated by $x$
then we have an isomorphism 
$$
C_{\GG_2}(a)\cong \Z_3\times\Z_3\times  C_3(a)\times C_4(\omega\phi)
$$
which is compatible with the conjugation action by $\omega^2$ if we let 
$\omega^2$ act on the right hand side by multiplication with $-1$ on $C_3$, on $C_4$ 
and on one of the factors $\Z_3$, and trivially on the other factor $\Z_3$. 

If $C:=C_{S_2}(a)$ denotes the centralizer of $a$ in $S_2$ then this isomorphism restricts to an 
isomorphism 
$$
C_3(a)\times \ZZ_3\times\ZZ_3\to C \ .
$$ 
If $N:=N_{S_2}(a)$ is the normalizer of $a$ in $S_2$ then we get isomorphisms 
$N\cong C\rtimes C_4(\omega^2)$ and  
$N_{\GG_2}(a)\cong C\rtimes Q_8$ with $Q_8$ the quaternion group 
of order $8$ generated by $\omega^2$ and $\omega\phi$. 

The action of $S_2$ on $\F_9[u^{\pm 1}]$ is trivial and therefore we get 
$$
H^\ast(C,\HV1) \cong H^\ast(C_3,\HV1) \otimes \La(a_1,\zeta_1) \cong \FF_9[y_1,u^{\pm 1}]
\otimes \La(x_1,a_1,\zeta_1)
$$
where $x_1 \in H^1(C_3,\FF_3)$ and $y_1  \in H^1(C_3,\FF_3)$ are the usual generators of the 
cohomology of a cyclic group of order $3$, and $a_1$ and $x_1$ are one dimensional 
classes accounting for the two factors $\Z_3$ of $C$.
Conjugation by $\omega$ defines an isomorphism
$$
\omega_\ast: H^\ast(C,\FF_3) \longr H^\ast(\omega C\omega^{-1},\FF_3)
$$
and we will write $x_2 = \omega_\ast x_1$ and so on. {\bf Note:} In what follows we use the
subscripts to distinguish between cohomology classes of the two centralizers, 
except possibly when referring to the class $u$.

The conjugation action of $\GG_2$ on the normal subgroup $S_2$ induces an action of 
$SD_{16}\cong \GG_2/S_2$ on $H^\ast(S_2,\HV1)$. Likewise the conjugation action of 
$N_{\GG_2}(a)$ on $C$ induces an action of the quotient group $Q_8$ on $H^\ast(C,\HV1)$. 
This action of $Q_8$ extends to an action of $SD_{16}$ on 
$H^\ast(C,\HV1)\times H^\ast(\omega C\omega^{-1},\HV1)$. 
The bedrock calculation is the following result from \cite{HennDuke}. (See also \S 4 of \cite{GHMR}.) 

There is an $SD_{16}$-linear monomorphism  
\begin{align}\label{rho-defined}
\rho:H^\ast(S_2,\FF_9[u^{\pm 1}]) &\ \longr  H^\ast(C,\FF_9[u^{\pm 1}])) \times 
H^\ast(\omega C\omega^{-1},\FF_9[u^{\pm 1}])\\
& \cong\Big(\prod_{i=1,2} \FF_3[y_i] \otimes \La(a_i,x_i,\zeta_i)\Big) \otimes \FF_9[u^{\pm 1}]\nonumber
\end{align} 
whose image is the free $\FF_9[y_1+y_2,u^{\pm 1}] \otimes \Lambda(\zeta_1 + \zeta_2)$-module
generated by the elements
\begin{equation}
\begin{array}{ccccc}\label{hwclasses}
1, & x_1, & x_2, & y_1, \\
x_1a_1-x_2a_2, & y_1a_1, & y_2a_2, & y_1x_1a_1& . 
\end{array}
\end{equation}
Note that $y_1$ could be replaced with $y_2$ and $y_1x_1a_1$ with $y_2x_2a_2$.

The action of $SD_{16}$ is determined by the formulas
(see (4.6) and (4.7) of \cite{GHMR}) 
\begin{equation}\label{GHMRFor1}
\begin{array}{cccccccccc}
\omega_*(x_1)&=&x_2, & \omega_*(y_1)&=&y_2, & \omega_*(a_1)&=&a_2\\ 
\omega_*(x_2)&=&-x_1, & \omega_*(y_2)&=&-y_1, & \omega_*(a_2)&=&-a_1 \\
\\
\phi_*(x_1)&=&-x_2, & \phi_*(y_1)&=&-y_2, & \phi_*(a_1)&=&-a_2\\
\phi_*(x_2)&=&-x_1, & \phi_*(y_2)&=&-y_1, & \phi_*(a_2)&=&-a_1 &\ .
\end{array}
\end{equation}
Furthermore,  
\begin{equation}\label{GHMRFor2}
\begin{array}{ccccccc}
\omega_*(u)&=& \omega u, & \omega_*(\zeta_1)&=& \zeta_2, & \omega_*(\zeta_2)=\zeta_1\\
\phi_*(u)&=& u, & \phi_*(\zeta_1)&=& \zeta_2, & \phi_*(\zeta_2)=\zeta_1\\
\end{array}
\end{equation} 
and $\omega_*$ acts by an $\FF_9$-linear algebra map while $\phi_*$ is an $\FF_3$-linear  
algebra map which is $\FF_9$-antilinear, i.e. $\phi_*(\lambda z)=\lambda^3\phi_*(z)$
if $\l\in \FF_9$. 

\subsection{The cohomology of $\GG_2$ and $\GG_2^0$ with coefficients in $\HV1$.} 
\label{coh-G2}  
After taking invariants with respect to the subgroup $D_8$ generated by $\omega^2$ 
and $\phi$ the map $\rho$ of (\ref{rho-defined}) induces a monomorphism  
$$
\widetilde{\rho}: H^\ast (\GG_2^0,\HV1) \cong H^\ast (S_2,\HV1)^{D_8}
\to \big(H^\ast(N,\HV1) \times H^\ast(\omega N \omega^{-1}, \HV1)\big)^{C_2(\phi)}
$$
In this subsection we will write elements in the target as couples. 
%The action of $\omega$ and $\phi$ 
%on such couples is determined by (\ref{GHMRFor1}) and (\ref{GHMRFor2}). This means that for 
%$\l_1,\l_2\in \FF_9$, $k\in \Z$ we have 
%\begin{equation}
%\begin{array}{cccccc} 
%\omega_*(\l_1u^k,\l_2u^k)&=&(\l_1\omega^ku^k,\l_2\omega^ku^k), 
%& \phi_*(\l_1u^k,\l_2u^k)&=&(\l_1^3u^k,\l_2^3u^k)\nonumber \\
%\omega_*(\l_1x_1,\l_2x_2)&=&(-\l_2x_1,\l_1x_2), 
%& \phi_*(\l_1,x_1,\l_2x_2)&=&(-\l_2^3x_1,-\l_1^3x_2)\\
%\end{array}
%\end{equation}
%and so on.  
%
Then the following elements  
\begin{equation}\label{invclasses} 
\begin{array}{cc}
v_2^{1/2}:= (u^{-4},u^{-4}) & w:  =  (\omega^2u^{-4},\omega^{-2}u^{-4}) \\ 
\zeta:  =  (\zeta_1,\zeta_2) & a_{35}:=(\omega a_1u^{-18},\omega^{-1}a_2 u^{-18}) \\ 
\alpha:=(\omega x_1u^{-2},\omega^{-1}x_2u^{-2}) 
& \beta:=(\omega^3y_1u^{-6},\omega^{-3}y_2u^{-6})\ . \\
\end{array} 
\end{equation} 
are easily checked to be $D_8$-invariant, and with the exception of $v_2^{1/2}$  
they are even $SD_{16}$-invariant. 
We have $\omega_*(v_2^{1/2})=-v_2^{1/2}$, 
so $v_2^{1/2}$ is still an eigenvector for the action of $SD_{16}$. Note also that $w^2=-v_2$. 
The reason for the power $u^{-18}$ as well as the 
name of the class $a_{35}$ will become clear in \S \ref{Hom-EN}. The elements $\a$ and $\b$ detect 
(up to sign) the images of the elements $\alpha_1$ and $\beta_1$ of $\pi_\ast S^0$. 

By using (\ref{GHMRFor1}) and (\ref{GHMRFor2})  
we can easily check that the target of $\widetilde{\rho}$ is the free
module over $\FF_9[\beta,v_2^{\pm 1/2}] \otimes \La(\zeta,\alpha,a_{35})$ on the class
$1:=(1,1)$. (Note that we need to twist the $\FF_9$-module structure on the second factor 
$H^\ast(\omega N \omega^{-1}, \HV1)$ by Frobenius for this to be true). 
It is also the free module over $\FF_3[\beta,v_2^{\pm 1/2}] \otimes \La(\zeta,\alpha,a_{35})$ 
on the classes $1$  and $w$. 
The second point of view is better adapated for understanding the residual action of 
$\omega$ on the $D_8$-invariants and for describing the image of $\tilde{\rho}$ after 
passing to those invariants.  

Furthermore, by (\ref{hwclasses}), (\ref{GHMRFor1}) and (\ref{GHMRFor2}) the image of 
$\widetilde{\rho}$ is the free module over $\FF_3[\beta,v_2^{\pm 1/2}] \otimes \La(\zeta)$ 
on the $8$ classes 
\begin{equation}\label{gens-for-s0}
\begin{array}{cccc} 
1:=(1,1)&\ \  \alpha &\ \    w\alpha  &\ \  w\beta \\  
\a a_{35} &\ \   \beta a_{35} &\ \  w\beta a_{35} &\ \   w\b\alpha a_{35} \ . \\  
\end{array}
\end{equation}
Thus, missing from the image is the free $\FF_3[v_2^{1/2}]\otimes\La(\zeta)$-submodule 
of the target of $\widetilde{\rho}$ generated by the $4$ classes 
$$
w,\ a_{35},\ wa_{35},\ w\alpha a_{35}\ .
$$

Note that all of the $8$ classes in (\ref{gens-for-s0}) are eigenvectors for the residual action of 
$\omega$ on the $D_8$-invariants. Together with the fact that 
$v_2^{1/2}$ is an eigenvector with eigenvalue $-1$ for this residual action 
this yields the following result which is an extension of Corollary 19 of \cite{GHM}.  

\begin{thm}\label{G2-cohom} 

1) The restriction map 
$$
H^*(\GG_2^0,\FF_9[u^{\pm 1}])\to \Big(H^*(N,\FF_9[u^{\pm 1}])\times 
H^*(\omega N\omega^{-1},\FF_9[u^{\pm 1}])\Big)^{C_2(\phi)} 
$$ 
is a monomorphism. Its target is isomorphic to 
$\FF_9[\beta,v_2^{\pm 1/2}] \otimes \La(\zeta,\alpha,a_{35})$ and its image   
is the free module over $\FF_3[\beta,v_2^{\pm 1/2}] \otimes \La(\zeta)$ on the $8$ classes 
$1$, $\a$, $w\a$, $w\b$, $\a a_{35}$, $\b a_{35}$, $w\b a_{35}$ and $w\b\a a_{35}$. 

2) The restriction map 
$$
H^*(\GG_2,\FF_9[u^{\pm 1}])\to \big(H^*(N,\FF_9[u^{\pm 1}])\times 
H^*(\omega N\omega^{-1},\FF_9[u^{\pm 1}])\big)^{C_2(\phi)\times C_2(\omega)}
$$ 
is a monomorphism. Its target is isomorphic to the free module over 
$\FF_3[\beta,v_2^{\pm 1}] \otimes \La(\zeta,\alpha,a_{35})$ generated by $1$ and $w$, and its image   
is the free module over $\FF_3[\beta,v_2^{\pm 1}] \otimes \La(\zeta)$ on the $8$ classes 
$1$, $\a$, $w\a$, $w\b$, $\a a_{35}$, $\b a_{35}$, $w\b a_{35}$ and $w\b\a a_{35}$. 

3)  The residual action of $\omega$ on $H^*(\GG_2^0,\HV1)$ induces an eigenspace decomposition 
$$
H^*(\GG_2^0,\FF_9[u^{\pm 1}])\cong H^*(\GG_2,\FF_9[u^{\pm 1}])\oplus 
v_2^{1/2}H^*(\GG_2,\FF_9[u^{\pm 1}]) 
$$
in which $H^*(\GG_2,\FF_9[u^{\pm 1}])$ is the $(+1)$ eigenspace and 
$v_2^{1/2}H^*(\GG_2,\FF_9[u^{\pm 1}])$ is the $(-1)$-eigenspace. \qed
\end{thm}

\subsection{The homotopy groups of $E^{hG_{24}}\wedge V(1)$}\label{Hom-EV} 
In this subsection we recall the calculation in \cite{GHM} and add some information on Toda brackets 
and $\a$-multiplications.

We have that $G_{24}= C_3 \rtimes Q_8$ where $C_3$ generated by our element
$a$ of order $3$ and $Q_8$ is generated by $\omega^2$ and $\omega\phi$. 
The element of order $3$ acts trivially on $\HV1$ and
$$
H^\ast(G_{24},\HV1) \cong H^\ast(C_3,\HV1)^{Q_8}.
$$
We have
$$
H^\ast (C_3,\HV1) \cong \FF_9[y_1,u^{\pm 1}] \otimes \La(x_1).
$$
The action of $\omega^2$ is $\FF_9$-linear; however, the action
of $\omega\phi $ is $\FF_9$-antilinear.

By abuse of notation we define  
%\begin{align}
\begin{equation}
\label{the-basics} 
\alpha = \omega x_1u^{-2}, \ \ \ \beta = \omega^3y_1u^{-6}, \ \ \ w=  \omega^2 u^{-4} 
\end{equation}
%\end{align}
of bidegrees $(1,4)$, $(2,12)$ and $(0,8)$ respectively. As before these elements
$\alpha$ and $\beta$ will detect (up to sign) the images of the
elements $\alpha_1$ and $\beta_1$ of $\pi_\ast S^0$.

Using the formulas (\ref{GHMRFor1}) we have
\begin{equation}\label{E2-G24-V(1)}
H^\ast(G_{24},\HV1) = \FF_3[\beta,w^{\pm1 }] \otimes \La(\alpha).
\end{equation} 

By the calculations of many people
(see \cite{GHMR} \S 3, for example), the spectrum $E^{hG_{24}}$ has
an invertible class $\Delta^3 \in \pi_{72}E^{hG_{24}}$; this class reduces
to $-w^9$. In the homotopy fixed point spectral sequence
$$
H^\ast(G_{24},\FF_9[u^{\pm 1}]) \Longrightarrow \pi_\ast E^{hG_{24}}\wedge V(1)
$$
all differentials commute with multiplication by $\alpha$, $\beta$,
and $w^9$ and are determined by 
\begin{align}\label{basic-diffs}
d_5(w^{i+3}) &=\pm \alpha\beta^2w^{i},\qquad 0 \leq i \leq 5\ ;\nonumber\\
\\
d_9(w^{i+6}\alpha) &=\pm \beta^5w^i,\qquad 0 \leq i \leq 2\ .\nonumber
\end{align}

In particular $w^i$, $0 \leq i \leq 2$, and $w^i\alpha$, $0 \leq i \leq 5$, are permanent
cycles.  The first of these differentials implies that the Toda bracket
\begin{equation}\label{whatisxi}
z_i=\langle \alpha, \alpha, \beta^2w^i\rangle,\qquad 0 \leq i \leq 2
\end{equation} 
is non-zero and detected by $\alpha w^{i+3}$ and a shuffle argument
with Toda brackets gives, perhaps up to sign,
\begin{equation}\label{a-mult}
\alpha z_i = \langle \alpha, \alpha, \alpha\rangle \beta^2w^{i} = 
\beta^3w^{i},\qquad 0 \leq i \leq 2.
\end{equation} 
The homotopy groups  $\pi_\ast E^{hG_{24}}\wedge V(1)$ are
$72$-periodic on $w^9 = \omega^2 v_2^{9/2}$. They are displayed graphically in Figure 1.

Since $\pi_{72+i}E^{hG_{24}}\wedge V(1)=0$ for $i=4$ and $i=5$,
the homotopy class $w^9$ can be extended to a map
$$
\Sigma^{72}V(1) \to E^{hG_{24}}\wedge V(1).
$$
Since $E^{hG_{24}}$ is a ring spectrum, this can be extended to an
equivalence
\begin{equation}\label{tmf-smash-v1-per}
\xymatrix{
\Sigma^{72}E^{hG_{24}}\wedge V(1) \rto^{\simeq}
& E^{hG_{24}}\wedge V(1).
}
\end{equation}
Thus we say $E^{hG_{24}}\wedge V(1)$ is $72$-periodic. Of course, it is a standard 
fact that $E^{hG_{24}}$ is itself $72$-periodic. For one reference among many, see
\cite{GHMR} \S 3.

We will also need the homotopy groups of $E^{hG_{12}} \wedge V(1)$, where
$G_{12} = C_3 \rtimes C_4(\omega^2)$. The 
$E_2$-term of the homotopy fixed point spectral sequence is
$$
H^\ast (G_{12},\HV1) = H^\ast(C_3,\HV1)^{C_4} = \FF_9[\beta,v_2^{\pm 1/2}] \otimes
\La(\alpha)
$$
and all the differentials are determined by (\ref{basic-diffs}). In fact, 
$$
\pi_\ast E^{hG_{12}} \wedge V(1) \cong
\pi_\ast (E^{hG_{24}} \wedge  V(1)) \otimes \FF_9
$$
and
$$
\pi_\ast E^{hG_{24}} \wedge V(1)
\cong [\pi_\ast E^{hG_{12}} \wedge V(1)]^{C_2(\phi\omega)}.
$$

\subsection{The homotopy groups of $\pi_\ast E^{hN} \wedge V(1)$}\label{Hom-EN}

Recall that $N$ denotes the normalizer of $a$ in $\SS_2$  
and that $N= C\rtimes C_4(\omega^2)$. By (\ref{GHMRFor1}) we have
$$
H^\ast(C,\HV1) \cong H^\ast(C_3,\HV1) \otimes \La(a_1,\zeta_1)
$$
with $(\omega^2)_\ast a_1 = -a_1$ and $(\omega^2)_\ast \zeta_1 = \zeta_1$. 
By continued abuse of notation we name the class $\omega^{-18}a_1$ of bidegree $(1,36)$ simply by 
$a_{35}$, $u^{-4}$ by $v_2^{1/2}$ and $\zeta_1$ by $\zeta$. 
With this and the notation used in \S \ref{Hom-EV} we have
$$
H^\ast (N,\FF_9[u^{\pm 1}])
\cong [H^\ast (C,\FF_9[u^{\pm 1}])]^{C_4(\omega^2)} \cong \FF_9[\beta,v_2^{\pm 1/2}] \otimes 
\Lambda  (\alpha,a_{35},\zeta) \ . 
$$ 
Furthermore, if $N^1 = \SS_2^1 \cap N$, then $N^1=(C_3\times \Z_3)\rtimes C_4(\omega^2)$ and 
$\omega^2$ acts by multiplication with $-1$ on $C_3$ and on $\Z_3$ and thus we have
$$
H^\ast (N^1,\FF_9[u^{\pm 1}]) \cong \FF_9[\beta,v_2^{\pm 1/2}] \otimes \Lambda 
(\alpha,a_{35})\ .
$$

The inclusion $G_{12} \subseteq N^1$ gives the map on cohomology defined by sending
$a_{35}$ to zero. The selection of degree 36 and the name $a_{35}$ is explained 
by the existence of a fiber sequence 
$$
\Sigma^{35} E^{hG_{12}} \longr E^{hN_1} \longr E^{hG_{12}}.
$$
which was established in Corollary 13 of \cite{GHM}. The maps in the short exact sequence
$$
0 \to a_{35}H^\ast(G_{12},\FF_9[u^{\pm 1}]) \to H^\ast(N^1,\FF_9[u^{\pm 1}])
\to H^\ast(G_{12},\FF_9[u^{\pm 1}]) \to 0
$$
are maps of $E_2$-terms of spectral sequences abutting to a short exact sequence
in homotopy groups
$$
0 \to \pi_\ast \Sigma^{35}E^{hG_{12}} \wedge V(1) \to \pi_\ast E^{hN_1}\wedge V(1) \to
\pi_\ast E^{hG_{12}} \wedge V(1)\to 0.
$$
This last statement is proved in Lemma 14 of \cite{GHM} by a computation; it is not
a conceptual result. We conclude that there is a splitting
\begin{equation}
\pi_\ast E^{hN^1}\wedge V(1)  \cong \pi_\ast (E^{hG_{12}}\wedge V(1))\otimes 
\Lambda(a_{35}).
\end{equation}
where we have confused $a_{35}$ with some choice of homotopy
class in $\pi_{35} E^{hN^1}$ detected by $a_{35}$.
This should {\it not} be interpreted as a
ring isomorphism, but only as an isomorphism of modules over
$\FF_9[\beta] \otimes \Lambda(\alpha)$.

To get the homotopy groups of $E^{hN}\wedge V(1)$ we use the
fiber sequence 
$$
\xymatrix{
E^{hN} \rto & E^{hN^1} \rto^{id-k} & E^{hN^1}
}
$$
where $k$ is a topological generator of $\ZZ_3 \cong N/N^1$.
By Proposition 17 of \cite{GHM} (again, a case-by-case calculation) we have that 
$$
0=
\pi_\ast(id-k):\pi_\ast E^{hN^1}\wedge V(1) \longr
\pi_\ast E^{hN^1}\wedge V(1)
$$
and we can conclude that there is an isomorphism of 
$\FF_3[\beta]\otimes\Lambda(\alpha)$-modules
\begin{equation}\label{N-smash-v1}
\pi_\ast E^{hN}\wedge V(1)\cong (\pi_\ast E^{hN^1}\wedge V(1))\otimes \Lambda(\zeta)\cong
\pi_\ast (E^{hG_{12}}\wedge V(1))\otimes \Lambda(a_{35},\zeta)\ .
\end{equation}

The spectrum $E^{hN}\wedge V(1)$ is periodic, although the argument
is not quite as simple as in (\ref{tmf-smash-v1-per}). We have the following result.

\begin{prop}\label{N-period} The spectra $E^{hN}\wedge V(1)$ and
$E^{hN^1}\wedge V(1)$ are $72$-periodic.
\end{prop}

\begin{proof} We do the case of $N$; the case $N_1$ follows the same line of
reasoning. To begin, we note $\pi_{72+4}E^{hN}\wedge V(1)=0$.
This is equivalent to noting that $\pi_jE^{hG_{24}}\wedge V(1)=0$
for $j=4,5,41$, and $42$, and this can be read off Figure 1. But note
$\pi_{72+5}E^{hN} \ne 0$ as $\pi_{43}E^{hG_{24}}\wedge V(1)\ne 0$.
However multiplication by $\beta$ is a monomorphism on $\pi_{72+5}E^{hN}$
and we can appeal to Proposition \ref{v1-maps}  
below to extend $v_2^{9/2}$ to a map 
$$
\Sigma^{72} V(1) \to E^{hN}\wedge V(1)\ . 
$$
This map can be then be extended to an equivalence of $E_2^{hN}$-module spectra. 
\end{proof}

\subsection{The homotopy groups of $E^{h\GG_2^0}\wedge V(1)$ and of $L_{K(2}V(1)$.}
\label{him-V1}  
The differentials and multiplicative extensions
in the spectral sequence
$$
H^\ast(\GG_2^0,\FF_9[u^{\pm 1}]) \Longrightarrow
\pi_\ast(E^{h\GG_2^0} \wedge V(1)) 
$$
are now forced by naturality and the homotopy fixed point spectral sequence of 
$\pi_\ast E^{hN} \wedge V(1)$. 

In fact, there is a diagram of spectral sequences 
$$
\xymatrix@C=15pt{
H^\ast(\GG_2^0,\HV1) \rto  \ar@2{->}[d] 
& [H^\ast (N,\HV1) \times H^\ast(\omega^{-1}N\omega,\HV1)]^{C_2(\phi)}
 \ar@2{->}[d] \\  
\pi_\ast (E^{h\GG_2^0}) \wedge V(1) \rto&
\pi_\ast (E^{hN} \vee E^{h\omega^{-1}N\omega})^{C_2(\phi)} \wedge V(1) \\
}
$$ 
and  the differentials on the left hand side are determined by 
(\ref{basic-diffs}) and (\ref{N-smash-v1}) (cf. \cite{GHM} \S 3). 

More precisely, by Theorem \ref{G2-cohom} the $E_2$-term of the right hand side spectral sequence 
is a free module over $\FF_9[\beta,v_2^{\pm 1/2}] \otimes \La(\zeta,\alpha,a_{35})$ 
on the class $1$. This $E_2$-term can also be considered as a free module over 
$\FF_3[\beta,v_2^{\pm 1/2}] \otimes \La(\zeta,\alpha,a_{35})$
on the classes $1$ and $w$. The $E_2$-term of the left hand side spectral sequence injects into 
this $E_2$-term as the free module over $\FF_3[\beta,v_2^{\pm 1/2}] \otimes \La(\zeta)$ on the $8$ 
classes $1$, $\a$, $w\a$, $w\b$, $\a a_{35}$, $\b a_{35}$, $w\b a_{35}$ and $w\b\a a_{35}$. 
Note that this is only a graded $\FF_3$-subvectorspace  and not an $\FF_9$-subvectorspace. 

The differentials on the right hand side are $\FF_9$-linear 
and are determined by (\ref{basic-diffs}) and the fact that $\alpha$, 
$\beta$ and $\zeta$ are obviously permanent cycles while $a_{35}$ is a permanent cycle by 
(\ref{N-smash-v1}). By using that $w=\omega^2v_2^{1/2}$ we get, up to sign, 
the following non-trivial differentials on the left hand side 

\begin{align}\label{G20diffs}
d_5(v_2^{(i+3)/2}) &=v_2^{(i-1)/2}\beta^2w\alpha,\ \ \ \ \ \ \ \ \ \ \ \  i\equiv 0,1,2,3,4,5\mod(9)\nonumber\\
d_5(v_2^{(i+3)/2}\beta a_{35})&=v_2^{(i-1)/2}\beta^2w\beta\alpha a_{35} ,
\qquad i\equiv 0,1,2,3,4,5\mod(9)\nonumber  \\
\\
d_5(v_2^{(i+3)/2}w\beta)&= v_2^{(i+1)/2}\alpha\beta^3,\qquad \ \ \ \ 
i+1\equiv 0,1,2,3,4,5\mod(9)\nonumber \\
d_5(v_2^{(i+3)/2}w\beta a_{35})&=v_2^{(i+1)/2}\alpha\beta^3a_{35},\qquad i+1\equiv 0,1,2,3,4,5\mod(9) \nonumber 
\end{align}
and
\begin{align}
d_9(v_2^{(i+6)/2}\alpha) &=\pm v_2^{i/2}\beta^5,\qquad \ \ \ \ \ \  i\equiv 0,1,2\mod(9) \nonumber\\
d_9(v_2^{(i+6)/2}\alpha a_{35}) &=\pm v_2^{1/2}\beta^5a_{35},\qquad \ \ i\equiv 0,1,2\mod(9) \nonumber\\ 
\\
d_9(v_2^{(i+5)/2}w\alpha) &=\pm v_2^{(i-1)/2}w\beta^5,\qquad \ \ \ \ i\equiv 0,1,2\mod(9) \nonumber\\
d_9(v_2^{(i+5)/2}w\beta \alpha a_{35}) &=\pm v_2^{(i-1)/2}w\beta^6 a_{35},
\qquad i\equiv 0,1,2\mod(9)\ .\nonumber
\end{align}
The final result is periodic with periodicity generator $v_2^{9/2}$
and are summarized in the following result in which the generators have been chosen such that they 
correspond to those given in the calculation of $\pi_*(L_{K(2)}V(1))$ in the main theorem of \cite{GHM}. 

\begin{thm}\label{Split-G20} 1) As a module over  
$S:=\FF_3[v_2^{\pm 9/2},\b]\otimes\La(\zeta)$ there is an isomorphism   
\begin{align}
\pi_*(E_2^{h\GG_2^0}\wedge V(1))\cong\  
& S/(\b^5)\{v_2^{l}\}_{l=0,1,5}\oplus S/(\b^3)\{v_2^{l}\a\}_{l=0,1,2,5,6,7}\nonumber\\
& \oplus\ S/(\b^4)\{v_2^{l}w\b\}_{l=0,4,5}
\oplus\ S/(\b^2)\{v_2^{l}w\a\}_{l=0,1,2,4,5,6}  \nonumber\\
& \oplus\ S/(\b^4)\{v_2^{l}\b a_{35}\}_{l=0,4,5}
\oplus\ S/(\b^3)\{v_2^{l}\a a_{35}\}_{l=0,1,2,5,6,7} \nonumber \\
& \oplus\ S/(\b^5)\{v_2^{l}w\b a_{35}\}_{l=0,4,5}
\oplus\ S/(\b^2)\{v_2^{l}w\b\a a_{35}\}_{l=0,1,2,4,5,6}  \ . \nonumber 
\end{align} 

2) There is an isomorphism of $\pi_*(L_{K(2}V(1))$-modules 
$$
\pi_*(E^{h\GG_2^0}\wedge V(1))\cong \pi_*(L_{K(2}V(1))\oplus v_2^{9/2}\pi_*(L_{K(2}V(1))  
$$
which is compatible with the splitting 
$$
E^{h\GG_2^0}\simeq (E^{h\GG_2^0})^{+}\vee (E^{h\GG_2^0})^{-}
=E^{h\GG_2}\vee (E^{h\GG_2^0})^{-}
$$ 
of $E^{h\GG_2^0}$ into the $(+1)$ and $(-1)$ eigenspectrum 
for the residual action of $C_2(\omega)$ on $E^{h\GG_2^0}$.  \qed 
\end{thm}

The final result is displayed graphically in Figure 2 which describes  
the homotopy groups of $E_2^{h\GG_2^1}\wedge V(1)$ including the $\a$-multiplications 
determined by (\ref{a-mult}). Those of $L_{K(2)}V(1)$ 
are obtained by taking the tensor product with the exterior algebra on $\zeta$.  
Notice that $\pi_*(L_{K(2}V(1))$ is $144$-periodic  with periodicity generator $v_2^9$. 
This, by itself, does not imply
that $L_{K(2)}V(1)$ itself is 144-periodic as there is an obstruction to extending 
$v_2^9$ over $\Sigma^{144}L_{K(2)}V(1)$. 
The homotopy group $\pi_{149}L_{K(2)}V(1) \cong \ZZ/3$ is generated by a
class $y$ which in $E_6$ of the spectral sequence is given by $\zeta a_{35}w^{11}z_0$, 
where $z_0$ is as in (\ref{whatisxi}). 
We note that $\beta y \ne 0$. See Example \ref{periodicity}
below for more on this point.

We can apply this discussion to identify the homotopy groups
of the Brown-Comenetz dual  of $L_{K(2)}V(1)$. Since $L_{K(2)}V(1)$ is the localization
of a type 2 complex, $M_2V(1) = L_{K(2)}V(1)$; hence
$$
\pi_\ast I_2V(1) \cong \pi_\ast IL_{K(2)}V(1) \cong
\Hom(\pi_{-\ast}L_{K(2)}V(1),\QQ/\ZZ).
$$

In the sequel we will return to our convention of writing $V(1)$ for $L_{K(2)}V(1)$, 
leaving the localization implicit. 

The following result is also in \cite{BehTop06}.

\begin{thm}\label{dual-of-v1} 1.) The $K(2)$-local
spectrum $V(1)$ is Brown-Comenetz self-dual on homotopy groups; more precisely, there
is an isomorphism
$$
\pi_\ast \Sigma^{-28}V(1) \cong \pi_\ast I_2V(1)
$$
of modules over $\FF_3[\beta_1] \otimes \Lambda(\alpha_1)$.

2.) There is also an isomorphism
$$
\pi_\ast \Sigma^{-22}V(1) \cong \pi_\ast I_2\wedge V(1)
$$
of modules over $\FF_3[\beta_1] \otimes \Lambda(\alpha_1)$.
\end{thm}

\begin{proof} The two statements are equivalent, since
$$
I_2V(1) \simeq F(V(1),I_2) \simeq F(V(1),S^0) \wedge I_2 \simeq
\Sigma^{-6}V(1) \wedge I_2.
$$
Statement (1) follows from an inspection of Figure 2. Indeed, there is an isomorphism
$$
\pi_\ast V(1) \cong \Lambda(\zeta) \otimes \pi_\ast
E^{h\GG_2^1} \wedge V(1)
$$
where $\Lambda(\zeta)$ is an exterior algebra on a class in degree $-1$.
The homotopy groups of  $E^{h\GG_2^1} \wedge V(1)$ are displayed
in Figure 2 of the appendix. Inspection of this chart shows there is
an isomorphism
$$
\pi_\ast \Sigma^{-28} V(1) \cong \Hom (\pi_\ast V(1),\QQ/\ZZ)
$$
with respect to which $1 \in \pi_0V(1)$ corresponds to the class
labelled 
$$
\zeta a_{35}w^{-7}\beta^5
$$
in the conventions explained at the beginning of the appendix. This
class is in $\pi_{28}V(1)$.
\end{proof}

\subsection{\bf Creating $v_2$-self maps.} We would now like to give a simple criterion for
creating $K(2)$-local equivalences out of $V(1)$. 
The first result, distilled from far more subtle arguments given in
\cite{BehP}, is about extending homotopy classes over $V(1)$; it
works in the stable category in general. Let $W$ be the cofiber
of $v_1:S^4 \to V(0)$ and let $g:S^5 \to W$ be the map with cofiber
the unlocalized $V(1)$.

\begin{prop}\label{v1-maps} Let $X$ be a spectrum, not necessarily
$K(2)$-local, and let $y \in \pi_nX$ be of order $3$. Suppose the Toda 
bracket
$\langle y, 3, \alpha\rangle$ contains zero. 
Then there is an extension of $y$ to a map
$$
\bar{y}:\Sigma^n W \longr X
$$
and the composite
$$
\xymatrix{
S^{n+15} \rto^\beta & S^{n+5} \rto^g & \Sigma^n W \rto^{\bar{y}} & X
}
$$
is null-homotopic. In particular, if the multiplication
$$
\beta:\pi_{n+5} X \to \pi_{n+15}X
$$
is injective,  then $y$ extends to a map $\Sigma^n V(1) \to X$
\end{prop}

\begin{proof} The only statement we need to prove is that
$\bar{y}g\beta =0$. This will follow from the fact that 
$\pi_{15}W=0$. The relevant part of the long exact sequence for
computing this group is
$$
\xymatrix{
\pi_{11}S^0 \rto^{v_1}& \pi_{15}V(0) \rto &\pi_{15}W \rto &\pi_{10}S^0
\rto^{v_1} & \pi_{14} V(0)
}
$$
The group $\pi_{15}V(0)$ is $\ZZ/3$ generated by
$\alpha_4=v_1^3\alpha$ modulo $3$. But $\pi_{11}S^0 \cong \ZZ/9$
generated by $\alpha_{3/2}$ and $\alpha_{3/2} = v_1^2\alpha$ modulo $3$.
The group $\pi_{10}S^0$ is generated by $\beta = \beta_1$ and
modulo $3$, $v_1\beta_1 \ne 0$; in fact $v_1\beta = \pm \alpha
\bar{\beta}$ where $\bar{\beta}$ maps to $\beta$ under the pinch map
$V(0) \to S^1$. 
\end{proof}

Proposition \ref{v1-maps} now extends to a result about equivalences out of $V(1)$.

\begin{thm}\label{recon-v1} Let $X$ be a $K(2)$-local spectrum with a chosen
isomorphism of twisted $\GG_2$-modules
$$
g: E_\ast V(1) \cong E_\ast X.
$$
Let $\iota_X \in H^0(\GG_2,E_0 X)$ be the image of the generator
of $H^0(\GG_2,E_0V(1))$ under the map induced by $g$. Suppose that
\begin{enumerate}

\item $\iota_X$ is a permanent cycle in the Adams-Novikov Spectral Sequence and 
detects an element of order $3$;

\item $\pi_4 X = 0$;

\item multiplication by $\beta$ is monomorphic on $\pi_5 X$.  
\end{enumerate}
Then $g$ can be realized by a $K(2)$-local equivalence
$$
f:V(1) \longr X.
$$
\end{thm}

\begin{proof} Let $S^0 \to X$ be the map of order $3$ detected by $\iota_X$. 
Then  Proposition \ref{v1-maps} applies and yields a map $f:V(1) \to X$
which, by construction, realizes $g$. Since $g$ is an isomorphism, 
$f$ is a weak equivalence, after localization.
\end{proof}

\begin{exam}\label{periodicity} Even before $K(2)$-localization,  the spectrum $V(1)$
has a $v_2^9$-self map \cite{BehP}; hence the localized $V(1)$
is $144$-periodic. As an application of Theorem
\ref{recon-v1} we can give a simple proof of the second fact. Consider the isomorphism 
$$
v_2^9:E_\ast  V(1) \to E_\ast  \Sigma^{-144} V(1).
$$
Thus we see $v_2^9$ extends to a self equivalence
$$
v_2^9:\Sigma^{144}V(1) \longr V(1)
$$
and the localized $V(1)$ is 144-periodic.
\end{exam}

\begin{rem}\label{homo-Sdet-v1} Theorem \ref{dual-of-v1} does  not
quite yield a weak equivalence $\Sigma^{-28}V(1) \simeq I_2  V(1)$,
as Theorem \ref{recon-v1} doesn't immediately apply.
We have that
\begin{align*}
E_\ast I_2 V(1) &\cong E_\ast \Sigma^{-6}V(1) \otimes_{E_\ast} E_\ast I_2\\
&\cong \Sigma^{-4}E_\ast V(1) \otimes_{E_\ast} E_\ast \langle\det\rangle \\
& \cong \Sigma^{4} E_\ast V(1) \cong \Sigma^{-28}E_\ast V(1)
\end{align*}
since $E_\ast \langle\det\rangle/(3,v_1) \cong \Sigma^8 E_\ast V(1)$ as
twisted $\GG_2$-modules. (See Lemma \ref{homo-Sdet-v1-extra} for more
details.)
Thus we have a candidate for a permanent
cycle in $H^0(\GG_2,E_{-28}V(1))$, but we'd need to check that  this is 
a permanent cycle and detects the element specified in the previous proof.

We will see in Corollary \ref{dual-v1-real} that indeed we have
$\Sigma^{-22}V(1) \simeq I_2\wedge V(1)$ and, hence,  $\Sigma^{-28}V(1) \simeq I_2  V(1)$.
\end{rem}

\section{The centralizer resolution}

The results of this section are a revisiting of work of the second author from \cite{HennRes}.
The aim is to be able to compare and contrast algebraic and topological information
for some of our calculations in the next section.

Let $\Z_3[[\GG_2^1]]$ be the completed group ring of the group $\GG_2^1$.  
By Theorem 10 of \cite{HennRes} there is an algebraic resolution of the trivial 
$\Z_3[[\GG_2^1]]$-module 
\begin{align}\label{cent-res}
0\to C_3\to C_2\to C_1\to C_0\to \Z_3\to 0 
\end{align}
with 

\begin{enumerate}
\item $C_0=\Z_3[[\GG_2^1/G_{24}]]$, 

\item $C_1=\ZZ_3[[\GG_2^1]]\otimes_{\ZZ_3[SD_{16}]}\chi 
\oplus \ZZ_3[[\GG_2^1]]\otimes_{\ZZ_3[G_{24}]}\tilde{\chi}$   

\noindent 
where $\chi$ is the $\Z_3[SD_{16}]$-module whose underlying $\Z_3$-module is $\Z_3$  
and on which $\omega$ and $\phi$ act by $-1$, and where $\tilde{\chi}$ is the $\Z_3[G_{24}]$ module  
whose underlying $\Z_3$-module is $\Z_3$  
and on which $\omega^2$ acts by sign and both $a$ and $\omega\phi$ act trivially,

\item $C_2= \ZZ_3[[\GG_2^1]]\otimes_{\ZZ_3[SD_{16}]} \lambda$ 

\noindent
where $\l$ is the $\Z_3[SD_{16}]$-module whose underlying $\Z_3$-module is $\WW$ 
and on which $\omega$ acts by multiplication with $\omega^2$ and $\phi$ acts by Frobenius, 

\item $C_3=\ZZ_3[[\GG_2^1/SD_{16}]]$. 

\end{enumerate}

We get a slightly larger resolution of $\ZZ_3$ as $\GG_2$-module by using the 
decomposition $\ZZ_3 \times \GG_2^1 \cong \GG_2$ of (\ref{GG2-splitting})  
and then tensoring the resolution
(\ref{cent-res}) with the small resolution
$$
0 \to \ZZ_3[[\ZZ_3]] \longr \ZZ_3[[\ZZ_3]] \longr \ZZ_3 \to 0\ .
$$

This leads to a small spectral sequence 
$$
E_1^{p,q}(\GG_2^1)=E_1^{p,q}=\Ext_{\Z_3[[\GG_2^1]]}^q(C_p,\FF_9[u^{\pm 1}]) 
\Longrightarrow H^{p+q}(\GG^1_2,\FF_9[u^{\pm 1}]) \ . 
$$ 

By Shapiro's lemma and by using that the modules $\chi$, $\tilde{\chi}$ and $\l$ are self-dual 
this $E_1$-term can be rewritten as  
\begin{align*}
E_1^{0,q} &\cong  H^q(G_{24},\FF_9[u^{\pm 1}])b_0\\
E_1^{1,q} &\cong \Ext_{\Z_3[G_{24}]}^q(\tilde{\chi},\FF_9[u^{\pm 1}])\oplus 
\Ext_{\Z_3[SD_{16}]}^q(\chi,\FF_9[u^{\pm 1}])\\
&\cong H^q(G_{24},\FF_9[u^{\pm 1}]\otimes\tilde{\chi})\oplus 
H^q(SD_{16},\FF_9[u^{\pm 1}]\otimes \chi) \\
& \cong H^q(G_{24},\FF_9[u^{\pm 1}])b_{36} \oplus H^q(SD_{16},\FF_9[u^{\pm 1}])e_8\\ 
E_1^{2,q} & \cong \Ext_{\Z_3[SD_{16}]}(\l,\FF_9[u^{\pm 1}])\cong 
H^q(SD_{16},\FF_9[u^{\pm 1}]\otimes \l) \\
&\cong H^q(SD_{16},\FF_9[u^{\pm 1}])e_{36} \oplus H^q(SD_{16},\FF_9[u^{\pm 1}])e_{44}\\ 
E_1^{3,q} & \cong H^q(SD_{16},\FF_9[u^{\pm 1}])e_{48}
\end{align*} 
where $Rc_n$ means a free module over $R$ on a generator of degree $n$. 
In fact, we have 
$$
H^*(SD_{16},\FF_9[u^{\pm 1}])\cong \FF_3[v_2^{\pm 1}]\ , 
$$ 
in particular this cohomology is concentrated in cohomological dimension $0$ and it is $16$-periodic.  
Hence the degree of the generator $e_{48}$ could be changed by 
any multiple of $16$. Likewise the degree of the generators $e_8$, $e_{36}$ and $e_{44}$ could be 
changed by any multiple of $16$. By (\ref{E2-G24-V(1)}) we have 
$$
H^*(G_{24},\FF_9[u^{\pm 1}])\cong \FF_3[w^{\pm 1},\beta]\otimes \La(\alpha)
$$ 
and the degrees of the generators $b_0$ and $b_{36}$ could be changed by  any multiple of $8$. 
We have made our choices in order to simplify some of the formulas below. Note that the element $b_{36}$ corresponds to the element $a_{35}$ in section 3. 

This spectral sequence must converge towards the result described in Theorem \ref{G2-cohom}.b 
with the exterior algebra on $\zeta$ removed. This implies that  the only non-trivial 
differentials in this spectral sequence are, up to sign,
\begin{align}\label{alg-diffs}
d_1(w^{2i+1}b_0) &=  v_2^ie_8\nonumber\\
d_1(v_2^ib_{36}) &= v_2^ie_{36}\\
d_1(w^{2i+1}b_{36}) &=  v_2^ie_{44}\nonumber\\
d_2(\alpha w^{2i+1} b_{36}) &=  v_2^i e_{48}\nonumber
\end{align}
Note that  $E_3^{p,q} = 0$ for $p > 1$. There is a similar spectral sequence
for $H^\ast(\GG_2,\FF_9[u^{\pm 1}])$ with 
$$
E_1(\GG_2) \cong E_1(\GG_2^1) \otimes \Lambda(\zeta)
$$
and for which all differentials are determined by the fact that this is part of a splitting of 
spectral sequences and $a_0 \otimes \zeta \in E(\GG_2)_{(1,0)}$ is a permanent cycle.

By Theorem 11 of \cite{HennRes} these algebraic resolutions have topological resolutions 
as sequences of spectra. Note that $E^{hSD_{16}}$ is $16$-periodic and as above 
we have altered the suspensions of the original reference \cite{HennRes} by a multiple of $16$
in order to simplify some formulas below.
$$ 
E^{h\GG_2^1}\rightarrow \begin{array}{c} E^{hG_{24}} \end{array}
\rightarrow \begin{array}{c} \Sigma^8 E^{hSD_{16}} \\
\vspace{.05in} 
\vee \\ \vspace{.05in} \Sigma^{36}E^{hG_{24}} \end{array} \rightarrow
\begin{array}{c} \Sigma^{36} E^{hSD_{16}} \\ \vspace{.05in} \vee \\
\vspace{.05in} \Sigma^{44} E^{hSD_{16}}  \end{array} \rightarrow 
\begin{array}{c}
\Sigma^{48} E^{hSD_{16}} \end{array}
$$
and
$$ 
L_{K(2)}S^0 \rightarrow \begin{array}{c} E^{hG_{24}} \end{array}
\rightarrow \begin{array}{c} \Sigma^8 E^{hSD_{16}} \\
\vspace{.05in} \vee \\ \vspace{.05in} E^{hG_{24}} \\ \vspace{.05in}
\vee \\ \vspace{.05in} \Sigma^{36}E^{hG_{24}} \end{array} \rightarrow
\begin{array}{c} \Sigma^{36} E^{hSD_{16}} \\ \vspace{.05in} \vee \\
\vspace{.05in} \Sigma^{8} E^{hSD_{16}} \\ \vspace{.05in} \vee \\
\vspace{.05in} \Sigma^{44} E^{hSD_{16}} \\ \vspace{.05in} \vee \\
\vspace{.05in} \Sigma^{36}E^{hG_{24}} \end{array} \rightarrow 
\begin{array}{c}
\Sigma^{48}E^{hSD_{16}} \\ \vspace{.05in} \vee \\ \vspace{.05in}
\Sigma^{36} E^{hSD_{16}} \\ \vspace{.05in} \vee \\ \vspace{.05in}
\Sigma^{44} E^{hSD_{16}} \end{array}\rightarrow
\begin{array}{c}
\Sigma^{48}E^{hSD_{16}}\ . \end{array}
$$
All compositions and all Toda brackets are zero modulo indeterminacy and, thus,
these resolutions can be refined to towers of fibrations with $E^{h\GG_2^1}$ or
$L_{K(2)}S^0$ at the top, as needed. Thus if $Y$ is a $K(2)$-local spectrum,
there is a tower of fibrations
\begin{equation}\label{sphere-tower}
\xymatrix{
Y \rto & Y_3 \rto & Y_2 \rto & Y_1 \rto& E^{hG_{24}}\wedge Y\\
\Sigma^{-4}F_4 \wedge Y \uto &\Sigma^{-3}F_3 \wedge Y  \uto
& \Sigma^{-2}F_2 \wedge Y  \uto &
\Sigma^{-1}F_1 \wedge Y\uto
}
\end{equation}
with
$$
F_1 = \Sigma^8 E^{hSD_{16}} \vee E^{hG_{24}} \vee  \Sigma^{36}E^{hG_{24}}
$$
and so on for $F_2$, $F_3$, and $F_4$.

This leads to a spectral sequence $E_1^{s,t} \Longrightarrow \pi_{t-s}Y$
with
\begin{align*}
E_1^{0,t} &= \pi_tE^{hG_{24}}\wedge Y\\
E_1^{1,t} &= \pi_t\Sigma^{8}E^{hSD_{16}} \wedge Y
\oplus\pi_tE^{hG_{24}}\wedge Y
\oplus \pi_t\Sigma^{36}E^{hG_{24}}\wedge Y\\
E_1^{2,t} &=\pi_t\Sigma^{36}E^{hSD_{16}} \wedge Y
\oplus\pi_t\Sigma^{8}E^{hSD_{16}}\wedge Y
\oplus \pi_t\Sigma^{44}E^{hSD_{16}} \wedge Y
\oplus\pi_t\Sigma^{36}E^{hG_{24}}\wedge Y\\
E_1^{3,t} &= \pi_t\Sigma^{48}E^{hSD_{16}} \wedge Y
\oplus \pi_t\Sigma^{36}E^{hSD_{16}} \wedge Y
\oplus\pi_t\Sigma^{44}E^{hSD_{16}}\wedge Y\\
E_1^{4,t} &= \pi_t\Sigma^{48}E^{hSD_{16}} \wedge Y
\end{align*}

\begin{rem}\label{diffs-short} While we won't use this information below, it is worth noting
that it is a simple matter to calculate the differentials in the spectral sequence of this 
tower, at least for $Y=V(1)$. If, as before, $N$ denotes the normalizer of our elemnt $a$ 
in $\SS_2$ then there is a 
resolution of the trivial $N$-module $\ZZ_3$ as
\begin{equation}\label{N-res}
0 \to \ZZ_3[[N]] \otimes_{\ZZ_3[G_{12}]} \tilde{\chi} \to
\ZZ_3[[N]] \otimes_{\ZZ_3[G_{12}]} \tilde{\chi}
\oplus \ZZ_3[[N/G_{12}]] \to \ZZ_3[[N/G_{12}]] \to \ZZ_3
\end{equation}
which translates into a short tower of fibrations
\begin{equation}\label{N-tower}
\xymatrix{
E^{hN} \wedge Y \rto & Z_1 \rto & E^{hG_{12}} \wedge Y\\
\Sigma^{34} E^{hG_{12}} \wedge Y \uto & \Sigma^{35} E^{hG_{12}} \wedge Y \vee 
\Sigma^{-1} E^{hG_{12}} \wedge Y \uto
}
\end{equation}
There is an obvious projection from the tower (\ref{sphere-tower}) to the tower (\ref{N-tower})
calculating the map $Y \to E^{hN} \wedge Y$; algebraically, this is given by the
evident projection on resolutions. In the case
$Y=V(1)$, the topological spectral sequence for $\pi_\ast E^{hN} \wedge V(1)$ collapses
by Propositions 14 and 17 of \cite{GHM}. This and the differentials of
(\ref{alg-diffs}) determine all differentials in the
topological spectral sequence for $V(1)$ itself. Notice that because there are differentials in
the homotopy fixed point spectral sequence
$$
H^\ast(G_{24},\HV1) \Longrightarrow \pi_\ast E^{hG_{24}}\wedge V(1)
$$
the homotopy gets assembled slightly differently than might be predicted by the spectral 
sequence above for $H^\ast(\GG_2^1,\FF_9[u^{\pm 1}])$.  
\end{rem}

Our main use of these topological resolutions is the following result.

\begin{thm}\label{more-v1s} Let $Y$ be a spectrum so that there are equivalences
$$
E^{hG_{24}}\wedge V(1) \simeq E^{hG_{24}} \wedge Y 
\qquad\mathrm{and}\qquad  E^{hSD_{16}}\wedge V(1) \simeq E^{hSD_{16}} \wedge Y.
$$
Let $\iota_Y \in \pi_0E^{hG_{24}}\wedge Y$ be the image of
$1 \in \pi_0E^{hG_{24}}\wedge V(1)$ under the isomorphism 
$$
\pi_0 E^{hG_{24}}\wedge V(1) \cong \pi_0E^{hG_{24}} \wedge Y.
$$
If $\iota_Y$ survives
to a class $\epsilon_Y:S^0 \to Y$ of order $3$,
then $\epsilon_Y$ extends to a $K(2)$-local equivalence
$$
\bar{\epsilon}_Y:V(1) \longr Y.
$$
\end{thm}

\begin{proof} We will use Proposition \ref{v1-maps}  to produce
a commutative square
$$
\xymatrix{
V(1)  \rto^{\bar{\epsilon}_Y} \dto & Y\dto \\
E^{hG_{24}} \wedge V(1) \rto_\simeq & E^{hG_{24}} \wedge Y
}
$$
where the bottom map is the given weak equivalence and both vertical maps are the natural
maps given by the unit $S^0 \to E^{hG_{24}}$. Since $E$ is a $E^{hG_{24}}$-module,
it will follow that $E_\ast \bar{\epsilon}_Y$ is an isomorphism and that $\bar{\epsilon}_Y$
is a $K(2)$-equivalence.

By hypothesis we have a commutative square
$$
\xymatrix{
S^0  \rto^{\epsilon_Y} \dto & Y\dto \\
E^{hG_{24}} \wedge V(1) \rto_\simeq & E^{hG_{24}} \wedge Y.
}
$$
Using the calculation of
$\pi_\ast E^{hG_{24}}$ displayed in Figure 1, the fact that $\pi_\ast E^{hSD_{16}} \wedge V(1)
\cong \FF_3[v_2^{\pm 1}]$ and the spectral sequence of the tower above, we see that
$\pi_4 Y = 0$ and $\pi_5 Y$ is either zero or isomorphic to $\ZZ/3$ on a class which
supports a non-zero $\beta$ multiplication. In the nomenclature above 
(with $b_{36}$ changed to $a_{35}$) this class is detected by 
$$
\zeta\a v_2^{-2}a_{35} 
$$
and both this class and $\zeta\a\b v_2^{-2}a_{35}$ must be permanent cycles. There
is a possibility, just based on degree considerations, that either of the two could be hit
by a differential; however we see that the only possible such differentials are
$$
d_1(\a\beta^iv_2^{-2}a_{35}) = c_i \zeta \a\beta^iv_2^{-2}a_{35},\qquad i=0,1
$$
for some $c_i \in \ZZ/3$. Then $\beta$-multiplication implies that $c_0 = c_1$. 
\end{proof}

\section{Exotic elements of the Picard group and $V(1)$} 

In the Picard group of weak equivalence classes of $K(2)$-local spectra there is a subgroup
$\kappa_2$ of elements $X$ so that $E_\ast X \cong E_\ast = E_\ast S^0$ as twisted
$\GG_2$-modules; these are the {\it exotic elements} of the Picard group.
In \cite{Picat3}, we computed this group to be $(\ZZ/3)^2$. This extended work of 
Kamiya and Shimomura (\cite{shiPic}).  Here we
briefly review those results and then discuss $X \wedge V(1)$ for various $X \in\kappa_2$.

If $X \in \kappa_2$, then we have an Adams-Novikov Spectral Sequence
$$
H^\ast (\GG_2,E_\ast) \mathop{\longr}_{\cong}^{f_\ast} H^\ast(\GG_2,E_\ast X) \Longrightarrow
\pi_\ast X
$$
where $f_\ast$ is determined by the chosen isomorphism $E_\ast \cong E_\ast X$. If 
we let $\iota_X$ be the image of $1 \in H^0(\GG_2,E_0)$, then $X\simeq L_{K(2)}S^0$
if and only if $\iota_X$ is a permanent cycle. Shimomura had noticed that the only
possible non-zero differential is $d_5$ and, in effect, calculated
$H^5(\GG_2,E_4)\cong (\ZZ/3)^2$. Kamiya and 
Shimomura further showed that some of the elements could
be realized as targets of differentials from an element in $\kappa_2$; we completed this
task. The following result is Remark 3.2 of \cite{Picat3}. 

\begin{lem}\label{reduction} The reduction map
$$
r_\ast:H^5(\GG_2,E_4) \longr H^5(\GG_2,\FF_9[u^{\pm 1}]_4)
$$
is an injection and the image is generated by the elements $v_2^{-2}\b^2w\a$ 
and $\zeta \beta v_2^{-3}\a a_{35}$. 
\end{lem}

Thus if $X \in \kappa_2$ we may write
\begin{align}\label{diff-for-1}
d_5(\iota_X) &= [c_1v_2^{-2}\b^2w\a + c_2 \zeta \beta v_2^{-3}\a a_{35}]\iota _X\\
&\defeq a_X \iota_X\nonumber
\end{align}
for some $c_1$ and $c_2$ in $\ZZ/3$. This expression is multiplicative in the following
sense. Suppose $X$ and $Y$ are in $\kappa_2$ and we have chosen isomorphisms
$E_\ast X \cong E_\ast \cong E_\ast Y$ as twisted $\GG_2$-modules. These choices
determine an isomorphism $E_\ast (X \wedge Y) \cong E_\ast$ and a multiplicative
pairing of spectral sequences
\begin{equation}\label{diff-for-2a}
\xymatrix{
H^\ast (\GG_2,E_\ast) \otimes H^\ast (\GG_2,E_\ast) \ar@{=>}[r]\dto_m &
\pi_\ast X \otimes \pi_\ast Y \dto^{\wedge}\\
H^\ast (\GG_2,E_\ast) \ar@{=>}[r] & \pi_\ast (X \wedge Y).
}
\end{equation}
where $m$ is the ring multiplication.  In particular. we have 
\begin{equation}\label{diff-for-2}
d_5(\iota_{X\wedge Y}) = (a_X + a_Y) \iota_{X\wedge Y}.
\end{equation}

The restriction maps define a commutative square
$$
\xymatrix{
H^5(\GG_2,E_4) \rto \dto & H^5(G_{24},E_4)\dto^{r_\ast}\\
H^5(\GG_2,\FF_9[u^{\pm 1}]_4) \rto & H^5(G_{24},\FF_9[u^{\pm1 }]_4)
}
$$
with the map $r_\ast: H^5(G_{24},E_4) \longr H^5(G_{24},\FF_9[u^{\pm 1}]_4)$ an isomorphism.
The source of $r_{\ast}$ is isomorphic to $\ZZ/3$ generated by $\alpha\beta^2\Delta^{-1}$, which
reduces to $\pm \alpha\beta^2w^{-3}=v_2^{-2}\b^2w\a$. 
Since $d_5\Delta = \pm \alpha\beta^2$,  
it follows that if $X \in \kappa_2$, then $X \wedge E^{hG_{24}} \simeq \Sigma^{24k} E^{hG_{24}}$,
with $k\in \Z$. Furthermore, because $E^{hG_{24}}$ is $72$-periodic on $\Delta^3$ we may assume 
$k\in \{0,1,2\}$. The assignment
$X \mapsto k\mod(3)$ defines short exact sequence
$$
0 \to \ZZ/3 \longr \kappa_2 \longr \ZZ/3 \to 0
$$
We call $X \in \kappa_2$ {\it truly exotic} if $k=0$; this is equivalent to having
$c_1=0$ in the formula (\ref{diff-for-1}). This sequence is a split; part of what this section
accomplishes is to provide a distinguished splitting.

\begin{prop}\label{v1-fail} Let $X \in \kappa_2$ be a non-trivial
truly exotic element in $\kappa_2$. Then $\pi_\ast X\wedge V(1)$ is not a free
module over $\Lambda(\zeta)$. In particular $\pi_\ast X$ cannot be isomorphic
to a shift of $\pi_\ast V(1)$.
\end{prop}

\begin{proof} 
We will use both the algebraic and topological spectral sequences
determined by the centralizer resolution. Since $E_\ast X \cong E_\ast$ as
twisted $\GG_2$-modules, the algebraic spectral sequence
for computing $H^\ast(\GG_2,E_\ast(X \wedge V(1))$ is isomorphic to
the spectral sequence for $H^\ast (\GG_1,E_\ast V(1))$. 
Thus we have exactly the differentials of (\ref{alg-diffs}). 
In particular  
$$
d_2(\alpha w^{2i+1} b_{36}) = \pm v_2^i e_{48}\ .
$$
In the topological spectral sequence given by the tower (\ref{sphere-tower})
we must have
$$
d_2(\iota_X) =  \pm \zeta \beta v_2^{-3}\a a_{35}\iota_X 
$$
This means that at $E_3$ the class
$$
y=\beta v_2^{-3}\a a_{35}\iota_X \in E_3^{1,\ast}
$$
is annihilated by $\zeta$. Since $E_3^{s,t} = 0$ for $s > 2$ and $t-s=-1$, we also have that $y$
is a permanent cycle and detects a homotopy class annihilated by $\zeta$. Since
$y \in E^{1,\ast}$ and is divisible by $\zeta$ at $E_1$, the homotopy class detected
by $y$ cannot be divisible by $\zeta$.
\end{proof}

\begin{lem}\label{v-1-success} There is a unique element $P \in \kappa_2$ so that 
$$
v_2^{3} \in H^\ast (\GG_2,E_\ast(P\wedge V(1))
$$
is a permanent cycle in the Adams-Novikov Spectral Sequence converging to
$\pi_\ast (P \wedge V(1))$.
\end{lem}

\begin{proof} There exist three elements $Y \in \kappa_2$ so that $v_2^3 \in
H^0(G_{24},\FF_9[u^{\pm 1}])$ is a permanent cycle. Fix one. We conclude that the only
possible non-zero differential is then
$$
d_5 (v_2^3\iota_Y) = c \zeta \b\a a_{35} \iota_Y
$$
where $c \in \ZZ/3$. Using the diagram of spectral sequences (\ref{diff-for-2a}) we then see that 
there is a unique truly exotic element $Q \in \kappa_2$ so that
$$
d_5 (v_2^3\iota_{Y \wedge Q}) = 0.
$$
Then $P = Y\wedge Q$.
\end{proof}

\begin{thm}\label{P-define} There is a unique element $P \in \kappa_2$ so that
there is a weak equivalence
$$
\Sigma^{48}V(1) \mathop{\longr}^{\simeq} P \wedge V(1).
$$
\end{thm}

\begin{proof} We first prove existence. Let $P$ the element of $\kappa_2$
identified in Lemma \ref{v-1-success}. Since $v_2^3$ is a permanent cycle
in $H^\ast(G_{24},E_\ast P)$ by restriction, we have that
$$
\Sigma^{48}E^{hG_{24}} \simeq E^{hG_{24}} \wedge P.
$$
It is also automatic that $E^{hSD_{16}} \wedge P \simeq E^{hSD_{16}}$
as the spectral sequence
$$
H^s (SD_{16},E_t P) \Longrightarrow \pi_{s-t}E^{hSD_{16}} \wedge P
$$
collapses to the $s=0$ line. Since $E^{hSD_{16}}$ is $16$-periodic,
we can now apply Theorem \ref{more-v1s} to the
spectrum $\Sigma^{-48}P$ to produce the equivalence
$V(1) \simeq \Sigma^{-48}P\wedge V(1)$.

To get uniqueness, notice that an equivalence
$\Sigma^{48}V(1) \to Y\wedge V(1)$ produces an isomorphism
$E_\ast \Sigma^{48} V(1) \cong E_\ast (Y \wedge V(1))$.
Since $H^0(\GG_2,\FF_9[u^{\pm 1}]_{48}) \cong \ZZ/3$ generated by $v_2^3$,
we have that $v_2^3$ is a permanent cycle in the Adams-Novikov spectral sequence 
for $Y\wedge V(1)$ and Lemma \ref{v-1-success} applies.
\end{proof}

\begin{rem}\label{shim-sim-results} In \cite{IchSm}, Ichigi and Shimomura calculated
$\pi_\ast X \wedge V(1)$ for some $X$ in the Picard group. These results are closely 
related to the calculations done in this section.
\end{rem}

\section{The calculation of $\Sdet\wedge V(1)$} 

In this section we construct a weak equivalence
$\Sigma^{72}V(1) \simeq \Sdet\wedge V(1)$.
We begin by reviewing the construction of $\Sdet$. This material is 
discussed in more detail in \S 2 of \cite{Picat3}.

The determinant map $\GG_2 \to \ZZ_3^\times$ of (\ref{det-defined})
defines a $\GG_2$-module
$\ZZ_3\langle\det\rangle$. If $M$ is any other $\GG_2$-module, we write
$$
M\langle\det\rangle = M \otimes_{\ZZ_3} \ZZ_3\langle\det\rangle\ .
$$ 

Define $\SGG_2\subseteq \GG_2$ to be the kernel of
$\det:\GG_2 \to \ZZ_3^\times$. The image of the central element 
$\psi ^4:=4+0.S\in \GG_2$ with respect to this homomorphism is a topological generator 
of the subgroup $1+3\ZZ_3$. 

Let $(E^{hS\GG_2})^{-}$ be the wedge summand
of $E^{h\SGG_2}$ realizing the $(-1)$-eigenspace of the action of $C_2 \subseteq \ZZ_3^\times$.
Then $\Sdet$ is defined as the fibre of the self map $\psi^{4}-\det(\psi^4)id$, i.e. there is a fibration   
\begin{equation}\label{sdet-def-seq}
\xymatrix{
\Sdet \rto &(E^{h\SGG_2})^{-}
\ar[rr]^{\psi^4 - \det(\psi^4)id} && (E^{h\SGG_2})^{-} 
}
\end{equation} 
from which one deduces the isomorphism $E_*(\Sdet)\cong E_*\langle\det\rangle$. 

We note that $\SGG_2$ is of index $2$ in the subgroup $\GG_2^1$ 
and therefore the residual action of the quotient group gives a splitting 
\begin{equation}\label{split-SG}
E_2^{h\SGG_2}\simeq (E_2^{h\SGG_2})^{+}\vee (E_2^{h\SGG_2})^{-}
\simeq E_2^{h\GG_2^1}\vee (E_2^{h\SGG_2})^{-}
\end{equation}
into its $(+1)$ eigenspectrum and its $(-1)$ eigenspectrum. 

Likewise the inclusion of the index $2$ subgroup $\GG_2^0\subset \GG_2$  
given as the kernel of the homomorphism  
$$
\xymatrix{
\GG_2 \rto^\det & \ZZ_3^\times \rto & \ZZ_3^\times/1+3\ZZ_3\cong C_2\ . 
}
$$ 
gives a residual action of $C_2$ on $E_2^{h\GG_2^0}$ and an associated splitting 
\begin{equation}\label{split-G20}
E_2^{h\GG_2^0}\simeq (E_2^{h\GG_2^0})^{+}\vee (E_2^{h\GG_2^0})^{-}\simeq 
E_2^{h\GG_2}\vee (E_2^{h\GG_2^0})^{-}
\end{equation}
into its $(+1)$ eigenspectrum and its $(-1)$ eigenspectrum.

\begin{prop}\label{split-off-0} Let $X$ be any $K(2)$-local spectrum so that the identity
of $X$ has order $3$. Then
$$
E^{h\GG_2^0} \wedge X \simeq X \vee (\Sdet \wedge X) \ . 
$$
and
$$
(E^{h\GG_2^0})^{-}\wedge X \simeq \Sdet \wedge X\ .
$$
\end{prop}

\begin{proof} We observe that the assumption on $X$ shows that  there is a fibration  
\begin{equation}\label{SdetX}
\xymatrix{\Sdet\wedge X \rto &
(E^{h\SGG_2})^{-}\wedge X
\ar[rr]^{(\psi^4 -id)\wedge 1_X} && (E^{h\SGG_2})^{-}\wedge X}
\end{equation} 
By using the splitting (\ref{split-SG}) and by taking the wedge sum of this fibration with the fibration 
$$
\xymatrix{X\simeq E^{h\GG_2}\wedge X  \rto &
E^{h\GG_2^1}\wedge X
\ar[rr]^{(\psi^4 -id)\wedge 1_X} && E^{h\GG_2^1}\wedge X}
$$ 
we get a fibration  
$$
\xymatrix{X\vee (\Sdet \wedge X) \rto &
E^{h\SGG_2}\wedge X
\ar[rr]^{(\psi^4 -id)\wedge 1_X} && E^{h\SGG_2}\wedge X}\ . 
$$ 
Because $\psi^4$ lies in the subgroup $\GG_2^0$ and its image in the quotient 
$\GG_2^0/\SGG_2\cong \GG_2/\GG_2^1\cong \ZZ_3$ is a topological generator 
we see that the fibre of the self map of $E^{\SGG_2}\wedge X$ given by $(\psi^4-id)\wedge id_X$ is 
also equivalent to $E^{h\GG_2^0}\wedge X$. This establishes the first part of the proposition.  

The second part is a consequence of (\ref{SdetX}) and the fact that $\Z_3^{\times}$ is an 
abelian group so that the actions of the two factors $C_2$ and $\Z_3$ on $E^{h\SGG_2}$ commute. 
\end{proof}

We noted in Remark \ref{id-p-v1} that the identity of $V(1)$ has order $3$ and so the proposition 
can be applied to $V(1)$.  Together with Theorem \ref{Split-G20} this immediately gives 
the following result. 

\begin{cor}\label{basic-shift} There is an isomorphism of
modules over $\FF_2[\beta,v_2^{\pm 9}] \otimes \Lambda(\alpha,\zeta)$
$$
\pi_\ast\Sigma^{72}V(1) \cong \pi_\ast \Sdet \wedge V(1)\ . \ \ \ \  \ \ \ \ \ \ \ \ \qed 
$$
\end{cor}

We need one more ingredient before we can identify $\Sdet\wedge V(1)$. 
Compare Remark \ref{homo-Sdet-v1}.

\begin{lem}\label{homo-Sdet-v1-extra} Multiplication by $v_2^{9/2}$ defines
an isomorphism of twisted $\GG_2$-modules
$$
\Sigma^{72} E_\ast V(1) \cong E_\ast (\Sdet \wedge V(1))\ .
$$
\end{lem}

\begin{proof} There is an isomorphism
$$
E_\ast \Sdet \cong E_\ast\langle\det\rangle = E_\ast \otimes \ZZ_3\langle\det\rangle.
$$
If we write $g=a + bS \in \SS_2$, then $\det(g) = a^4$ modulo $3$. Furthermore,
in $E_\ast V(1) \cong \FF_9 [u^{\pm 1}]$, $g$ acts trivially on degree $0$ and
$gu = au$. Thus, in $E_\ast (\Sdet \wedge V(1))$,
$$
g(u^{8i+4})= a^{8i+8}u^{8i+4} = u^{8i+4}
$$
as $a \in \FF_9^\times$ modulo $3$. In particular $v_2^{9/2} = u^{-36}$ is invariant, as needed.
\end{proof}

\begin{thm}\label{SdetV(1)} There is a weak equivalence
of $K(2)$-local spectra
$$
\Sigma^{72} V(1) \simeq \Sdet \wedge V(1).
$$
\end{thm}

\begin{proof} This follows from Theorem \ref{recon-v1}, Proposition \ref{basic-shift}, and
Lemma \ref{homo-Sdet-v1-extra}.
\end{proof}

\section{The Brown-Comenetz dual of the sphere}

The results of this section complete the proof of Theorem \ref{BCofS}.
Let $P \in \kappa_2$ be the exotic element of the Picard
group singled out in Theorem \ref{P-define}.

\begin{thm}\label{BCofSproof} Let $n=2$ and $p=3$. 
The Brown-Comenetz dual $I_2$  of $L_{K(2)}S^0$ is given as 
$$
I_2\simeq S^2\wedge \Zdet \wedge P\ . 
$$ 
\end{thm} 

\begin{proof}  We know from the Gross-Hopkins formula (\ref{GrHopFor})
and \cite{Picat3}, that
$$
I_2\simeq S^2\wedge \Zdet \wedge P^{\wedge n} \wedge Q
$$
where $0 \leq n \leq 2$ and $Q \in \kappa_2$ has the property that
$E^{hG_{24}} \wedge Q \simeq E^{hG_{24}}$. Combining Theorems
\ref{P-define} and \ref{SdetV(1)} we have
$$
S^2\wedge \Zdet \wedge P^{\wedge n} \wedge V(1) \simeq 
\Sigma^{48n + 74} V(1).
$$
>From this and Proposition \ref{v1-fail}, we can conclude that if $Q$ is
non-trivial, then 
$\pi_\ast I_2 \wedge V(1)$ is not free over $\Lambda(\zeta)$, which contradicts
Theorem \ref{dual-of-v1}. Finally, again using Theorem
\ref{dual-of-v1}, we need
$$
48n + 74 \equiv -22\qquad \mathrm{modulo}\ 144
$$
as $V(1)$ is 144-periodic. Thus $n=1$.
\end{proof}

\begin{cor}\label{dual-v1-real} There is a weak equivalence of $K(2)$-local
spectra
$$
\Sigma^{-22}V(1) \simeq I_2 \wedge V(1).
$$
\end{cor}

\begin{proof} Using Theorems \ref{P-define}, \ref{SdetV(1)}, and
\ref{BCofSproof}, we have
$$
I_2 \wedge V(1) \simeq S^2 \wedge \Sdet \wedge P \wedge V(1)
\simeq \Sigma^{2+72+48} V(1)
\simeq \Sigma^{-22}V(1)
$$
as $V(1)$ is $144$-periodic.
\end{proof}

%Theorems \ref{BCofSproof}, \ref{P-define}, and
%\ref{SdetV(1)} now give the following
% refinement of Proposition \ref{dual-of-v1}. {\bf This is actually proved
% earlier now, no?}
%
%\begin{cor}\label{dual-of-v1}There is are equivalences of $K(2)$-local
%spectra
%$$
%I_2\wedge V(1) \simeq \Sigma^{-22}V(1)
%$$
%and
%$$
%I_2V(1) = F(V(1),I_2) \simeq \Sigma^{-28}V(1).
%$$
%\end{cor}

\section{Appendix: The homotopy groups of $E^{h\GG_2^1}\wedge
V(1)$}

We display two charts, distilled from \cite{GHM}. In both charts we
adhere to the following conventions:

\begin{enumerate}

\item Each dot represents a group isomorphic to $\ZZ/3$.

\item The horizontal scale is by degree in homotopy groups; the
vertical scale is cosmetic only.

\item Both spectra are periodic; the bold part of
the  chart is one copy of the basic pattern in homotopy; the lighter 
parts are the periodic copies.

\item Horizontal line segments represent multiplication by $\beta = \beta_1$; 
the diagonal line segments represent multiplication by $\alpha = \alpha_1$.

\item The class $w = \omega^2v_2^{1/2}$ has the property that $w^2 = -v_2$; thus,
only odd powers of $w$ appear.

\item The leading term of each pattern is labeled with a name which
arises from the calculations in the cohomology of various subgroups
of $\GG_2$; see \cite{GHM} and section 3 for details. 

\item The element $1$ has degree $0$, $\alpha$
has degree $3$, $\beta$ had degree $10$, $w^{n}$ has degree
$8n$, $v_2^n$ has degree $16n$, and $a_{35}$ has degree $35$. But note that $a_{35}$ itself does
not appear as a homotopy class. As an example, in figure 2 the element $a_{35}w^{11}\beta$
has degree $35 + 8\cdot 11 + 10 = 133$ and the pattern on that
element extends to degree $173$.
\end{enumerate}

The first chart displays the 
homotopy groups of $E^{hG_{24}}\wedge V(1)$. This chart can also
be used to read off $\pi_\ast E^{hG_{12}}\wedge V(1)$ as
\begin{equation}\label{12from24}
\pi_\ast E^{hG_{12}}\wedge V(1) \cong \FF_9 \otimes
\pi_\ast E^{hG_{24}}\wedge V(1).
\end{equation}
There is one basic pattern, beginning on $1$, repeated three
times. This spectrum is $72$-periodic on the class $w^9=\omega^2v_2^{9/2}$.

\includegraphics[bb= 0 650 338 800]{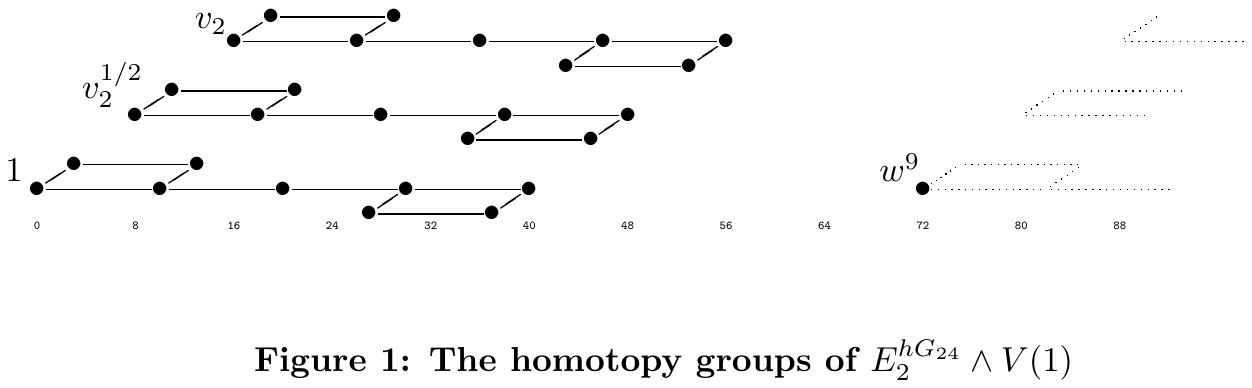}

The second chart displays the homotopy of $E^{h\GG_2^1}\wedge V(1)$.
There are
two basic patterns, one beginning on $1$, the other
on $w\alpha$. To each of these two patterns there is also a
dual pattern visible on $a_{35}w\beta$ and $a_{35}\alpha$
respectively. This spectrum is $144$-periodic on the class $v_2^9$.

The homotopy of $V(1)$  itself can be recovered from
$$
\pi_\ast V(1) = \Lambda(\zeta) \otimes
\pi_\ast E^{h\GG_2^1}\wedge V(1)
$$
where $\Lambda(\zeta)$ is the exterior algebra on a class of degree $-1$. 
\vfill\eject

\includegraphics[width=8.5in,angle=270]{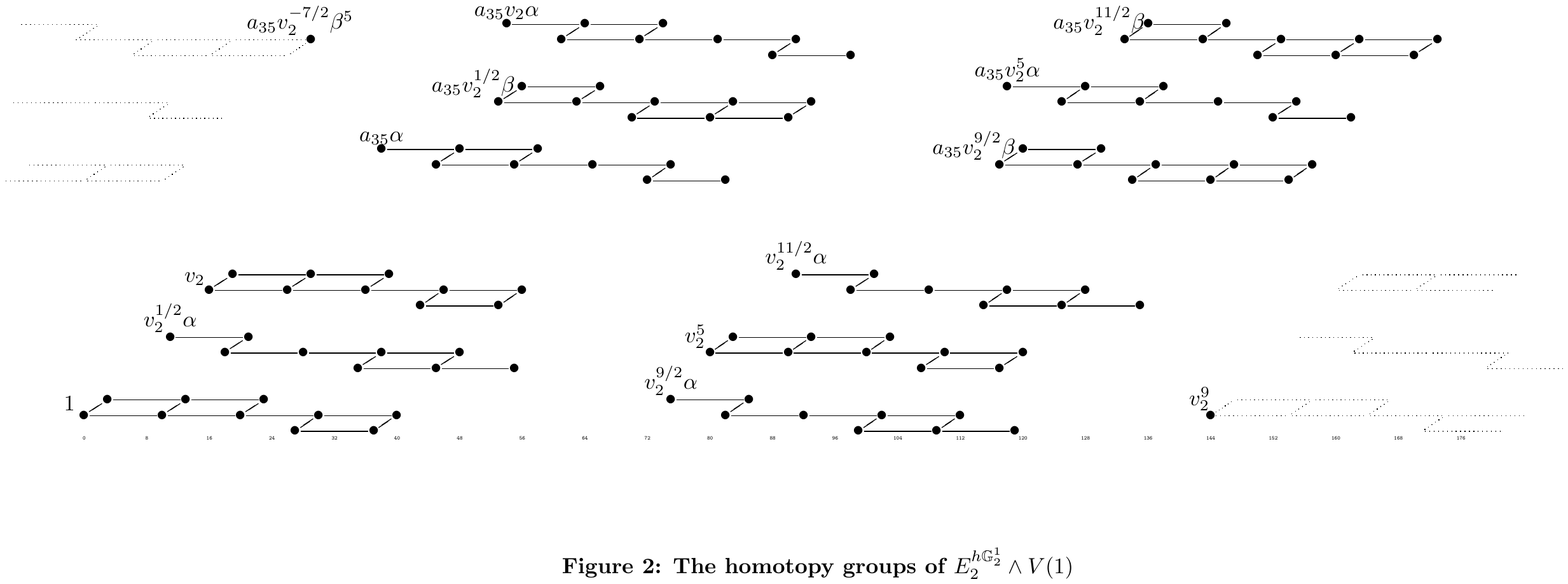}

\bibliographystyle{amsplain}
\input{bibliobc}

\bigbreak

\end{document}

%% file: macrospicard.tex
\newtheorem{thm}{Theorem}[section]
\newtheorem{cor}[thm]{Corollary}
\newtheorem{lem}[thm]{Lemma}
\newtheorem{prop}[thm]{Proposition}
\theoremstyle{definition}	

\newtheorem{rem}[thm]{Remark}

\newtheorem{exam}[thm]{Example}

\def\defeq{\overset{\mathrm{def}}=}

\def\goth{\mathfrak}

\def\a{\alpha}
\def\b{\beta}

\def\l{\lambda}

\def\Ga{\Gamma}
\def\La{\Lambda}

\def\FF{\mathbb{F}}
\def\GG{\mathbb{G}}

\def\QQ{\mathbb{Q}}

\def\SS{\mathbb{S}}

\def\Z{\mathbb{Z}}
\def\ZZ{{{\Z}}}

\def\Aut{\mathrm{Aut}}
\def\End{{{\mathrm{End}}}}
\def\Ext{\mathrm{Ext}}
\def\Gal{\mathrm{Gal}}
\def\Hom{\mathrm{Hom}}

\def\lim{\mathrm{lim}}

\def\Pic{{\mathrm{Pic}}}
\def\longr{{{\longrightarrow\ }}}

\def\Zdet{{{S^0\langle\det\rangle}}}
\def\Sdet{\Zdet}

\def\det{{{\mathrm{det}}}}
\def\SGG{{{\mathrm{S}\GG}}}

\newcommand{\F}{\mathbb{F}}

%% file: bibliobc.tex
\bibliographystyle{amsplain}
\bibliography{bibghm}
\bigbreak
\bigbreak